\documentclass[11pt,a4paper]{article}
\usepackage[utf8]{inputenc}
\usepackage[english]{babel}
\usepackage{graphicx}
\usepackage{float}
\usepackage{pdfpages}
\usepackage{amsthm}
\usepackage{url}
\usepackage{xcolor}
\usepackage{subcaption}
\usepackage[font=small,labelfont=rm]{caption}
\newtheorem{theorem}{Theorem}[section]
\newtheorem{lemma}{Lemma}[section]
\newtheorem{proposition}[theorem]{Proposition}

\theoremstyle{definition}

\theoremstyle{definition}

\newenvironment{remark}{%
	\par\noindent\textbf{Remark:}%
}{\par}

\usepackage{authblk}
\usepackage{amsmath}
\usepackage{amsfonts}
\usepackage{amssymb}

\usepackage{algorithmic}
\usepackage[linesnumbered,ruled,vlined]{algorithm2e}
\numberwithin{equation}{section}
\usepackage{mathtools}
\usepackage{scalefnt}
\usepackage{empheq}
\usepackage[a4paper,left=3.5cm,right=3cm,top=3cm,bottom=3cm]{geometry}
\usepackage{makeidx}
\usepackage{bm}
\makeindex

\usepackage{indentfirst}

\title{Well-posedness of nonlinear parabolic equations with unbounded drift via nonlinear evolution theory}

\author[1]{Thi Tam Dang\thanks{Corresponding author. Email addresses: \texttt{tam.dang@helsinki.fi} (Thi Tam Dang), \texttt{trung-hau.hoang@matfyz.cuni.cz} (Trung Hau Hoang), \texttt{giandomenico.orlandi@univr.it} (Giandomenico Orlandi), \texttt{tuomo.valkonen@iki.fi} (Tuomo Valkonen)}}
\author[2]{Trung Hau Hoang}
\author[3]{Giandomenico Orlandi}
\author[1,4]{Tuomo Valkonen}

\affil[1]{Department of Mathematics and Statistics, University of Helsinki, Finland}
\affil[2]{Faculty of Mathematics and Physics, Charles University, Czech Republic}
\affil[3]{Dipartimento di Informatica, Universit\`a di Verona, Italy}
\affil[4]{MODEMAT Research Center in Mathematical Modeling and Optimization, Quito, Ecuador}

\date{}

\begin{document}

\maketitle

\begin{abstract}
We develop a nonlinear evolution framework for nonlinear parabolic equations with unbounded drift terms formulated in Lorentz spaces. The main contribution lies in the construction of uniformly $m$-accretive operators based on Lorentz--Sobolev embeddings, which allows us to apply the Crandall--Liggett generation theorem for nonlinear evolution equations. Within this framework, we establish existence, uniqueness, and stability of mild solutions. We further show that these mild solutions coincide with weak solutions, ensuring consistency with the variational formulation. Finally, we investigate the long-time asymptotic behavior of solutions.
\end{abstract}

\section{Introduction}
In this paper, we study the following nonlinear parabolic problem:
\begin{equation}\label{eq:main}
\begin{cases}
u_t - \operatorname{div}\!\big( A(x,t,\nabla u) + B(x,t,u) \big) = -\operatorname{div} F & \text{in } \Omega_T,\\
u = 0 & \text{on } \partial\Omega \times (0,T),\\
u(0) = u_0 & \text{in } \Omega,
\end{cases}
\end{equation}
where $\Omega \subset \mathbb{R}^N$ with $N \ge 2$ is a bounded domain with sufficiently regular boundary, and $\Omega_T := \Omega \times (0,T)$. Moreover, 
\begin{itemize}
\item $F \in L^2(\Omega_T)$ and $u_0 \in L^2(\Omega)$.

\item The operator $A:\Omega_T \times \mathbb{R}^N \to \mathbb{R}^N$ is assumed to be a Carath\'eodory function satisfying the following growth and monotonicity conditions: there exist constants $\alpha,\beta>0$ such that, for all $\eta,\eta^* \in \mathbb{R}^N$,
\begin{align}
|A(x,t,\eta)| &\le \beta |\eta| + g(x,t), \quad g \in L^2(\Omega_T), \label{eq:Abound}\\
\langle A(x,t,\eta) - A(x,t,\eta^*), \eta - \eta^* \rangle &\ge \alpha |\eta - \eta^*|^2. \label{eq:Acoer}
\end{align}

\item The operator $B:\Omega_T \times \mathbb{R} \to \mathbb{R}^N$ is a Carath\'eodory function satisfying, for all \(z, z^* \in \mathbb{R}\),
\begin{align}
|B(x,t,z) - B(x,t,z^*)| &\le b(x,t)\,|z - z^*|, \label{eq:Bbound}\\
B(x,t,0) &= 0, \label{eq:B0}
\end{align}
where $b \in L^\infty(0,T; L^{N,\infty}(\Omega))$.
\end{itemize}

When the drift term is bounded (encoding advection through $\mathrm{div}( b \cdot\nabla u )$), problem~\eqref{eq:main} recovers the homogeneous Fokker--Planck equation describing Brownian motion evolution~\cite{Porretta01}. Related works~\cite{Cardaliaguet, Mos-Por, Poretta99, Poretta15} solve problem \eqref{eq:main} assuming either $b\equiv 0$ (pure diffusion) or $b\in L^\infty(0,T;L^1(\Omega))$ with growth condition \eqref{eq:Bbound}. However, the analysis of nonlinear parabolic equations with rough or unbounded drifts remains a central topic in modern PDE theory, with applications ranging from Fokker--Planck dynamics to semiconductor transport and diffusion in heterogeneous media. In such settings, the lack of sufficient regularity of the drift field prevents the direct application of classical $m$-accretive operator theory. The recent contributions~\cite{FM18, Farroni-Greco, Farroni25, Farroni23, Farroni26} provided a first framework for treating drifts in truncated Lorentz spaces, yet relied on restrictive assumptions and lacked a characterization of the long-time regime.

In \cite{FM18}, the authors prove only the existence of solutions to problem \eqref{eq:main} under the stronger assumption that the coefficient $b$ belongs to a proper convex subset $X(\Omega_T)$ of  $L^\infty(0,T;L^{N,\infty}(\Omega))$. In contrast, this paper establishes existence and uniqueness for \eqref{eq:main} under general drift conditions, thereby relaxing the restrictive assumption imposed in \cite{FM18}. Our approach leverages the general theory of nonlinear evolution families generated by time-dependent $m$-accretive operators, employing a dynamic truncation scheme $M_k\to\infty$ to construct explicit mild solutions via resolvent iteration, 
with uniform a priori estimates ensuring Aubin--Lions compactness and resolvent-Gronwall uniqueness in the limit.

The explicit nonlinear evolution family formulation obtained in this work naturally accommodates applications to numerical approximation, optimal control, and the study of long-time behavior. These aspects are typically difficult to capture through classical Galerkin-based methods. In particular, our results cover nonlinear Fokker--Planck and semiconductor-type models of physical relevance, extending the applicability of the semigroup framework beyond previous analytical limits.

To the best of our knowledge, this is the first approach based on nonlinear evolution theory that provides both existence and uniqueness, as well as long-time asymptotic behavior, for nonlinear parabolic equations with an unbounded drift term, where the associated coefficient belongs to a Lorentz space (see Section~\ref{sub:lorentz} for Lorentz spaces). This space admits functions with stronger singularities than those in \( L^{N,q}(\Omega) \) for \(1 \le q \le \infty \), but it does not retain compactness or several other desirable analytical features. The classical theory of evolution operators \cite{HR81,Rob01,Tem98} requires \( m \)-accretivity of the full spatial operator, which fails under drift conditions in \( L^{N,\infty}(\Omega) \). To overcome this difficulty, we first decompose the unbounded drift term \( b(x,t) \) into two components: one that can be bounded and a remainder that can be controlled in an appropriate way. We then introduce the truncated time-dependent operator \( \tilde{A}_{M_k}(t) \). Next, we prove that the truncated operators \( \tilde{A}_{M_k}(t) \) are uniformly \( m \)-accretive (Proposition~\ref{prop:m-accretivity}), thereby generating nonlinear evolution families \( \{S_{M_k}(t,s)\} \). Passing to the limit as \( M_k \to \infty \) via compactness yields a limiting evolution family \( S(t,s) \), which solves the full problem while preserving the quantitative stability estimates obtained from the approximations.

Furthermore, we investigate the long-time behavior of solutions to problem~\eqref{eq:main}. Porretta~\cite{Poretta15} analyzed Fokker--Planck equations with measure-valued drifts, establishing long-time averaging but not pointwise convergence. In contrast, we identify the global attractor as the singleton $\{u_\infty\}$, where $u_\infty \in W_0^{1,2}(\Omega)$ uniquely solves the stationary problem
\[
-\operatorname{div}\big[A(x,\nabla u_\infty)+B(x,u_\infty)\big] = -\operatorname{div} F.
\]
Moreover, we obtain explicit exponential convergence
\[
  \|u(t) - u_\infty\|_{L^2(\Omega)} \le M e^{-\omega t}\|u_0 - u_\infty\|_{L^2(\Omega)}, \quad t \ge 0,
\]
where $\omega = \alpha / (2C_P) > 0$ and $C_P \le \operatorname{diam}(\Omega)^2 / \pi^2$ is the Poincaré constant (see Section~\ref{sec:time-behavior} for more details). This rate coincides, up to the coercivity parameter $\alpha$, with the principal eigenvalue of the linear heat equation, confirming its optimality. These advances yield the \emph{first complete existence--uniqueness theory} for nonlinear parabolic equations with general $L^{N,\infty}(\Omega)$ drifts, substantially extending~\cite{FM18} while preserving their physical relevance to Fokker--Planck and semiconductor models. While~\cite[Theorem~4.2]{FM18} establishes $L^2$-decay under the integrability assumption $h \in L^1_{\mathrm{loc}}(\Omega)$, we prove exponential convergence to the physical steady state $u_\infty$. This provides, to the best of our knowledge, the sharpest available description of the long-time behavior for drifts in the Lorentz space. Furthermore, our Lyapunov functional $y(t) = \|u(t) - u_\infty\|^2$ allows us to identify the global attractor, a feature that was not addressed in the aforementioned work.

The paper is organized as follows. Section~\ref{sec:pre} presents the functional analytic framework, including Lorentz spaces (Section~\ref{sub:lorentz}) and $m$-accretivity of the truncated operators with their evolution families (Section~\ref{sub:accre}). Section~\ref{sec:main} states the main result (Theorem~\ref{thm:main}) and proves existence-uniqueness for problem~\eqref{eq:main}. Section~\ref{sec:time-behavior} analyzes the long-time behavior, establishing exponential convergence to steady states and identifying the global attractor.

\section{Preliminaries}\label{sec:pre}
This section provides an overview of Lorentz spaces and time-dependent $m$-accretive operators with their associated nonlinear evolution families, providing the analytical framework for our main results. For further details on Lorentz spaces, we refer to \cite{LG, Lorentz, ONeil}. For the theory of $m$-accretive operators and nonlinear evolution equations, see \cite{CL80, Kato, Kobayasi, HR81}.
\subsection{Lorentz spaces}\label{sub:lorentz}
For a measurable function $\varphi$ on $\Omega$, we define its distribution function as 
\[
\mu_\varphi(m) = |\left\{ x \in \Omega : |\varphi(x)| > m \right\}|.
\]
The Lorentz space $L^{p,q}(\Omega)$ with $1 \leq p,q < \infty$ consists of all measurable functions $\varphi:\Omega \to \mathbb{R}$ satisfying
\[
\|\varphi \|_{p,q} = \left( p \int_0^\infty \left[ \mu_\varphi(m) \right]^{p/q} m^{q-1}\, dm \right)^{1/q} < \infty.
\]
Equipped with this quasi-norm, $L^{p,q}(\Omega)$ forms a Banach space. In case $p=q$, we recover the Lebesgue space $L^p(\Omega)$. For $q=\infty$, the weak-$L^p$ space $L^{p,\infty}(\Omega)$ consists of functions where
\[
\|\varphi \|_{p,\infty} = \sup_{m>0} m \left[ \mu_\varphi(m) \right]^{1/p} < \infty.
\]
The spaces satisfy the continuous chain of embeddings
\[
L^r(\Omega) \hookrightarrow L^{p,q}(\Omega) \hookrightarrow L^{p,r}(\Omega) \hookrightarrow L^{p,\infty}(\Omega) \hookrightarrow L^q(\Omega)
\]
whenever $1 \leq q < p < r \leq \infty$.\\
A Hölder-type inequality in Lorentz spaces states that if $\varphi \in L^{p_1,q_1}(\Omega)$, $\psi \in L^{p_2,q_2}(\Omega)$ with $1<p_1,p_2<\infty$, $1\leq q_1,q_2\leq\infty$, then
\[
\|\varphi \, \psi \|_{p,q} \leq \| \varphi \|_{p_1,q_1} \|\psi\|_{p_2,q_2},
\]
where $\frac{1}{p} = \frac{1}{p_1} + \frac{1}{p_2}$ and $\frac{1}{q} = \frac{1}{q_1} + \frac{1}{q_2}$.\\
For $\varphi \in L^{p,\infty}(\Omega)$, the distance to bounded functions is defined by
\[
\mathrm{dist}(\varphi, L^\infty(\Omega)) = \inf_{ \psi \in L^\infty(\Omega)} \|\varphi- \psi \|_{L^{p,\infty}(\Omega)} = \lim_{n \to\infty} \|G_n \varphi \|_{p,\infty},
\]
where $G_n \varphi  = \varphi - T_n \varphi  $ and $T_n \varphi = \max\{-n, \min\{\varphi, n\}\}$ is truncation at level $\pm n>0$.\\
The Sobolev embedding extended to Lorentz spaces is stated in Theorem~\ref{thm:embed}.
\begin{theorem}[Sobolev embedding~\cite{Lorentz, ONeil}]\label{thm:embed}
Let $1<p<N$ and $1\leq q\leq p$. If $\varphi \in W_0^{1,1}(\Omega)$ with $|\nabla \varphi | \in L^{p,q}(\Omega)$, then $\varphi \in L^{p^*,q}(\Omega)$ where $p^* = \frac{Np}{N-p}$, and
\[
\|\varphi \|_{p^*,q} \leq S_{N,p} \|\nabla \varphi \|_{p,q}.
\]
Here $S_{N,p} = \omega_N^{-1/N} \frac{p}{N-p}$ with $\omega_N$ is the volume of the unit ball in $\mathbb{R}^N$.
\end{theorem}
\subsection{Accretive operators}
\label{sub:accre}
Let $(X,\|\cdot\|)$ be a uniformly convex Banach space. For each $t\in[0,T]$, let
$A(t):X\rightrightarrows X$ be a (possibly multivalued) operator with  $\operatorname{graph}(A(t))\subset X\times X$. We say that $A(t)$ is \emph{accretive} if
\[
\langle y_1-y_2, J(x_1-x_2)\rangle \ge 0,
\quad\forall (x_i,y_i)\in\operatorname{graph}(A(t)),\ i=1,2,
\]
where $J:X\to X^*$ is the normalized duality map. The family $\{A(t)\}_{t\in[0,T]}$ is said to be
\emph{uniformly $m$-accretive} if for every $t\in[0,T]$:
\begin{itemize}
\item $A(t)$ is accretive,
\item $\operatorname{ran}(I+\lambda A(t))=X$ for all $\lambda>0$,
\item the domains $D(A(t))$ are independent of $t$.
\end{itemize}
For each $\lambda>0$ and $t\in[0,T]$, the \emph{resolvent}
\[
J_\lambda^{A(t)}:=(I+\lambda A(t))^{-1}:X\to X
\]
is single-valued and nonexpansive:
\[
\|J_\lambda^{A(t)}x - J_\lambda^{A(t)}y\|\le \|x-y\|\quad\forall x,y\in X.
\]
When the dependence $t\mapsto A(t)$ satisfies appropriate continuity assumptions (see Theorem~\ref{thm:CL}), the family $\{A(t)\}$ generates a nonlinear evolution family on $X$.

\begin{theorem}[Crandall--Liggett for time-dependent operators~\cite{CL80}]\label{thm:CL}
Let $\{A(t)\}_{t\in[0,T]}\subset 2^{X\times X}$ be a family of nonempty sets in the Banach space $X$ such that $A(t)+wI$ is accretive for some fixed $w\in\mathbb{R}$. Assume that
\begin{enumerate}
\item $D(A(t)):=D(A(0))$ is independent of $t$,
\item $R(I+\lambda A(t))\supset\overline{D(A(0))}$ for $0<\lambda\le\lambda_0$, and $t\in[0,T]$,
\item  $|A(t)x|\le |A(\tau)x| + |t-\tau|L(\|x\|)(1+|A(\tau)x|)$ for $x\in D(A(0))$,
\item $\|(I+\lambda A(t))^{-1}x - (I+\lambda A(\tau))^{-1}x\|\le |t-\tau|L(\|x\|+|A(\tau)x|)$.
\end{enumerate}
Here $L:[0,\infty)\to[0,\infty)$ is increasing. Then for every $x\in D(A(0))$,
\[
S(t)x := \lim_{n\to\infty}\Bigl(I + \frac{t}{n}A\bigl(\tfrac{t}{n}\bigr)\Bigr)^{-n}x
\]
exists and defines a continuous function $t\mapsto S(t)x$.
\end{theorem}

\subsection{Nonlinear evolution problems}
\label{sub:evolution-problems}
We briefly recall the framework of nonlinear evolution operators following Kobayasi, Kobayashi, and Oharu~\cite{Kobayasi}. Let $X$ be a real Banach space, and let $\{A(t)\}_{t \in [0,T]}$ be a family of (possibly multivalued) nonlinear operators $A(t): X \rightrightarrows X$. We consider the non-autonomous Cauchy problem
\begin{equation}\label{eq:abstract}
\begin{aligned}
&u'(t) + A(t)u(t) = f(t),\quad t\in(0,T),\\
&u(0)=u_0,
\end{aligned}
\end{equation}
where \( u_0 \in X\) and $\{A(t)\}_{t\in[0,T]}\subset 2^{X\times X}$ is a family of operators satisfying the hypotheses of Theorem~\ref{thm:CL}.

For the specific case $X=L^2(\Omega)$, $D(A(t))=V=W_0^{1,2}(\Omega)$, $V^* = W^{-1,2}(\Omega)$, we assume that the family satisfies the following 
\begin{itemize}
\item[(H1)] \emph{Uniform accretivity}: For all \(u,v \in V\) and \(t \in [0,T]\),
  \[
  \langle A(t)u - A(t)v,J(u-v)\rangle_{V^*, V} \ge \alpha\|u-v\|_V^2,\quad \alpha>0.
  \]
\item[(H2)] \emph{Range condition}: $R(I+\lambda A(t))=X$ for all $\lambda>0$, $t\in[0,T]$.
\item[(H3)] \emph{Resolvent continuity}:  For $0 < \lambda \le \lambda_0$, $x \in X$, and $t,\tau \in [0,T]$,
  \[
  \|J_\lambda^{A(t)}x-J_\lambda^{A(\tau)}x\|_X\le|t-\tau|L(\|x\|_X).
  \]
\end{itemize}

\begin{theorem}[\cite{Kobayasi}]\label{thm:evolution}
Under assumptions (H1)--(H3) and Theorem~\ref{thm:CL}, there exists a unique \emph{evolution operator} $\{S(t,s)\}_{0\le s\le t\le T}:X\to X$ satisfying
\begin{enumerate}
\item $S(s,s)=\operatorname{Id}$,
\item $S(t,s)S(s,r)=S(t,r)$ for $r\le s\le t$,
\item $t\mapsto S(t,s)x$ is continuous in $X$ for fixed $s,x\in X$,
\item $\|S(t,s)x-S(t,s)y\|_X\le e^{L(t-s)}\|x-y\|_X$ locally,
\item $S(\cdot,s)x\in C([s,T];X)$.
\end{enumerate}
The mild solution of \eqref{eq:abstract} is given by
\[
u(t)=S(t,0)u_0 + \int_0^t S(t,r)f(r)\,dr.
\]
\end{theorem}

\section{Main Result}
\label{sec:main}
This section presents our main result (Theorem~\ref{thm:main}), which establishes the existence and uniqueness of solutions to problem~\eqref{eq:main} under the appropriate assumptions.

\begin{theorem}\label{thm:main}
Assume \eqref{eq:Abound}--\eqref{eq:B0} hold and $u_0\in L^2(\Omega)$. Then problem \eqref{eq:main} admits a unique  solution
\begin{equation*}
u \in L^2(0,T;W^{1,2}_0(\Omega)) \cap C([0,T];L^2(\Omega)),
\end{equation*}
satisfying
\begin{equation}\label{eq:weak}
\begin{aligned}
&-\int_{\Omega_T} u\,\partial_t \phi \,dx\,dt+ \int_{\Omega_T} \langle A(x,t,\nabla u) + B(x,t,u), \nabla \phi \rangle \,dx\,dt \\
&\qquad=  \int_{\Omega_T} \langle F, \nabla \phi \rangle \,dx\,dt+ \int_\Omega u_0 \phi(0) \,dx
\end{aligned}
\end{equation}
for all $\phi \in C_c^\infty(\Omega_T)$.
\end{theorem}
The proof of Theorem~\ref{thm:main} proceeds in seven steps. We first establish that the truncated operators $\tilde{A}_{M_k}(t)$ are uniformly $m$-accretive and generate nonlinear evolution families $\{S_{M_k}(t,s)\}$ (Proposition~\ref{prop:m-accretivity} and Proposition~\ref{prop:evolution-family}). Next, in Step~2, we formulate the corresponding approximate non-autonomous evolution equations~\eqref{eq:approx} and establish the existence of unique global mild solutions $u_k$ via a fixed-point argument in the variation-of-constants formulation~\eqref{eq:mild-approx}. Step~3 derives uniform a priori estimates for $\{u_k\}_{k\in\mathbb{N}}$. In Step~4, we apply compactness arguments to extract a convergent subsequence. Step~5 is devoted to passing to the limit as $k \to \infty$ in the mild formulation to obtain a limiting mild solution. In Step~6, we prove uniqueness of the solution. Finally, we identify the mild solution as a weak solution to the original problem~\eqref{eq:main}.

\begin{proof}The proof is organized into seven steps, as detailed below:

\textbf{Step 1: Nonlinear evolution generation.} \\
Let $T_{M_k}$ denote the standard truncation operator at levels $\pm M_k$. We define the truncation weight
\[
\theta_{M_k}(x,t):=\frac{T_{M_k}\big(b(x,t)\big)}{b(x,t)}, \quad x\in\Omega,\ t\in[0,T].
\]
For each $k\in\mathbb{N}$, consider the time-dependent operator $\tilde{A}_{M_k}(t):W_0^{1,2}(\Omega)\to W^{-1,2}(\Omega)$ given by
\begin{equation}\label{eq:truncated-operator}
\tilde{A}_{M_k}(t)u := -\operatorname{div}\big[A(x,t,\nabla u)+(1-\theta_{M_k}(x,t))B(x,t,u)\big].
\end{equation}
Equivalently, for all $u,v \in W_0^{1,2}(\Omega)$, the operator $\tilde{A}_{M_k}(t)$ is defined by
\[
\langle \tilde{A}_{M_k}(t)u, v \rangle
= \int_\Omega \big[ A(x,t,\nabla u)
+ (1-\theta_{M_k}(x,t))B(x,t,u) \big]\cdot \nabla v \, dx.
\]
The truncation is chosen in such a way that
\begin{equation}\label{eq:bbound}
\sup_{t\in[0,T]}\big\|b(\cdot,t)-T_{M_k}(b(\cdot,t))\big\|_{L^{N,\infty}(\Omega)} \le \frac{\alpha}{2S_{N,2}}.
\end{equation}
\begin{proposition}[Time-dependent $m$-accretivity]\label{prop:m-accretivity}
For any $t\in[0,T]$, let 
\[
\tilde{A}_{M_k}(t): L^2(\Omega) \supset D(\tilde{A}_{M_k}(t)) = W_0^{1,2}(\Omega) \to W^{-1,2}(\Omega),
\]
be defined by \eqref{eq:truncated-operator}. Then $\{\tilde{A}_{M_k}(t)\}_{t\in[0,T]}$ is uniformly $m$-accretive in $L^2(\Omega)$, i.e., it satisfies: 
\begin{enumerate}
\item[(i)] \emph{Accretivity}: For all \( t \in [0,T]\) and  \(  u,v\in W_0^{1,2}(\Omega) \), 
\[
\langle \tilde{A}_{M_k}(t)u - \tilde{A}_{M_k}(t)v, u-v \rangle_{W^{-1,2}(\Omega), W_0^{1,2}(\Omega)} \ge \frac{\alpha}{2}\|\nabla(u-v)\|_{L^2(\Omega)}^2.
\]

\item[(ii)] \emph{Range condition:} $R(I + \lambda \tilde{A}_{M_k}(t)) = L^2(\Omega)$ for all $\lambda>0$, $t\in[0,T]$.

\item[(iii)] \emph{Uniformly bounded resolvents:} for every $\lambda>0$ and $t\in[0,T]$, 
\[
\|J_\lambda^{\tilde{A}_{M_k}(t)}\|_{L^2(\Omega)\to L^2(\Omega)} \le 1.
\]
\end{enumerate}
\end{proposition}

\begin{proof}
(i) \emph{Accretivity} .  Let $w=u-v$. We have
\begin{equation*}
\begin{aligned}
\langle \tilde{A}_{M_k}(t)w, w \rangle_{W^{-1,2}(\Omega), W_0^{1,2}(\Omega)}& = \int_\Omega \big[A(x,t,\nabla u)-A(x,t,\nabla v)\big]\cdot\nabla w\,dx \\
&\qquad - \int_\Omega (1-\theta_{M_k}(x,t)) [ B(x,t, u)- B(x,t,v)] \cdot\nabla w\,dx.
\end{aligned}
\end{equation*}
The assumption \eqref{eq:Acoer} on $A$ implies that
\[
\int_{\Omega}\big[ A(x,t, \nabla u) - A(x,t ,\nabla v) \big] \cdot \nabla w \, dx \ge \alpha \|\nabla w\|_{L^2(\Omega)}^2.
\]
For the second term, the assumption \eqref{eq:Bbound} gives $|B(x,t, u)-B(x,t, v)|\le b(x,t )|w|$. Thus, we have
\[
\left|\int_{\Omega} (1-\theta_{M_k}(x,t))[B(x,t, u)-B(x,t,v)]\cdot\nabla w \, dx \right|  \le \int_{\Omega} | b-T_{M_k} b | |w||\nabla w| \, dx.
\]
Using the Sobolev embedding theorem (Theorem~\ref{thm:embed}) in Lorentz spaces, we arrive at
\begin{equation*}
\begin{aligned}
\int_{\Omega} | b-T_{M_k} b | |w||\nabla w| \, dx &\le \|b-T_{M_k} b\|_{L^{N,\infty}}\|w\|_{L^{2^*,2}(\Omega)}\|\nabla w \|_{L^2(\Omega)} \\
&  \le \|b-T_{M_k}b\|_{L^{N,\infty}(\Omega)} S_{N,2} \|\nabla w\|_{L^2(\Omega)}^2.
\end{aligned}
\end{equation*}
 Truncation choice \eqref{eq:bbound} gives
\[
\langle \tilde{A}_{M_k}(t)w,w\rangle \ge \frac{\alpha}{2}\|\nabla w\|_{L^2(\Omega)}^2.
\]

\medskip
\noindent (ii) \emph{Range condition}.
Fix $\lambda>0$ and $g\in L^2(\Omega)$. In order to prove that
\[
R(I+\lambda \tilde{A}_{M_k}(t)) = L^2(\Omega),
\]
we make use of a Galerkin approximation. To this end, let $\{w_j\}_{j=1}^\infty$ be an orthonormal basis of $W_0^{1,2}(\Omega)$, and for $n\in\mathbb{N}$, set $V_n=\operatorname{span}\{w_1,\dots,w_n\}$. We seek $u_n = \sum_{j=1}^n \xi_j^n w_j \in V_n$ in such a way that
\begin{equation}\label{eq:galerkin-resolvent}
\int_\Omega u_n \varphi_n\,dx 
+ \lambda \int_\Omega 
\big[ A(x,t,\nabla u_n) + (1-\theta_{M_k}(x,t)) B(x,t,u_n) \big]\cdot \nabla\varphi_n\,dx
= \langle g,\varphi_n\rangle_{L^2(\Omega)}
\end{equation}
for all $\varphi_n \in V_n$.\\
For any $u,\varphi \in W_0^{1,2}(\Omega)$, we define
\[
a_{M_k}(u,\varphi)
= \int_\Omega \big[ A(x,t,\nabla u) + (1-\theta_{M_k}(x,t))\, B(x,t,u) \big]\cdot \nabla \varphi \, dx.
\]
By the assumptions~\eqref{eq:Abound}--\eqref{eq:B0} on $A$ and $B$, the map $a_{M_k}$ is continuous on $W_0^{1,2}(\Omega)$ and satisfies
\[
a_{M_k}(u,u) \ge \alpha \|\nabla u\|_{L^2(\Omega)}^2, \qquad \forall\, u\in W_0^{1,2}(\Omega).
\]
We consider the bilinear form
\[
\mathcal{A}_{M_k,\lambda}(u,\varphi)
= \langle u,\varphi\rangle_{L^2(\Omega)} + \lambda\, a_{M_k}(u,\varphi),
\quad u,\varphi \in V_n.
\]
 For all $u_n\in V_n$, we have
\begin{equation*}
\begin{aligned}
\mathcal{A}_{M_k,\lambda}(u_n,u_n)&= \|u_n\|_{L^2(\Omega)}^2 + \lambda a_{M_k}(u_n,u_n)\\
& \ge \|u_n\|_{L^2(\Omega)}^2 + \lambda\alpha \|\nabla u_n\|_{L^2(\Omega)}^2 \\
& \ge c\,\|u_n\|_{W_0^{1,2}(\Omega)}^2,
\end{aligned}
\end{equation*}
with $c = \min(1,\lambda\alpha)>0$. Thus, $\mathcal{A}_{M_k,\lambda}$ is coercive and continuous on $V_n$, and by the Lax--Milgram theorem, there exists a unique $u_n\in V_n$ satisfying \eqref{eq:galerkin-resolvent}.\\
Choosing $\varphi_n=u_n$ as a test function in \eqref{eq:galerkin-resolvent}, we obtain
\[
\|u_n\|_{L^2(\Omega)}^2 + \lambda a_{M_k}(u_n,u_n)
= \langle g,u_n\rangle
\le \tfrac{1}{2}\|g\|_{L^2(\Omega)}^2 + \tfrac{1}{2}\|u_n\|_{L^2(\Omega)}^2.
\]
Hence,
\[
\|u_n\|_{W_0^{1,2}(\Omega)} \le C \|g\|_{L^2(\Omega)},
\]
where the constant $C$ is independent of \(n\).\\
Thus, $(u_n)$ is bounded in $W_0^{1,2}(\Omega)$, and, up to a subsequence, it holds that
\begin{align*}
&u_n \rightharpoonup u \text{ in } W_0^{1,2}(\Omega),\\
&u_n \to u \text{ in } L^2(\Omega),
\end{align*}
for some $u\in W_0^{1,2}(\Omega)$.\\
Passing to the limit in \eqref{eq:galerkin-resolvent} and using the monotonicity of $\tilde{A}_M(t)$, we obtain
\[
\int_\Omega u\,\varphi\,dx
+ \lambda a_{M_k}(u,\varphi)
= \langle g,\varphi\rangle,
\quad \forall\, \varphi\in W_0^{1,2}(\Omega),
\]
that is,
\[
(I+\lambda \tilde{A}_{M_k}(t))u = g.
\]
Therefore $R(I+\lambda \tilde{A}_{M_k}(t)) = L^2(\Omega)$, as claimed.

\medskip
\noindent (iii) \emph{Uniformly bounded resolvents.}
Let $u=J_\lambda^{\tilde{A}_{M_k}(t)}(x)$, $v=J_\lambda^{\tilde{A}_{M_k}(t)}(y)$. Then 
\[
u=x- \lambda\tilde{A}_{M_k}(t)u, \quad v=y-\lambda\tilde{A}_{M_k}(t)v.
\]
It follows that
\begin{align*}
\|u-v\|_{L^2(\Omega)}^2
&= \left\langle x-y-\lambda\big(\tilde{A}_{M_k}(t)u-\tilde{A}_{M_k}(t)v\big),\,u-v\right\rangle \\
&\le \|x-y\|_{L^2(\Omega)}\,\|u-v\|_{L^2(\Omega)},
\end{align*}
which implies
\[
\|u-v\|_{L^2(\Omega)}\le \|x-y\|_{L^2(\Omega)}.
\]
This completes the proof of the proposition.
\end{proof}

For fixed $k\in\mathbb{N}$, we consider the family of uniformly $m$-accretive operators
\[
\{\tilde{A}_{M_k}(t)\}_{t\in[0,T]}:L^2(\Omega)\rightrightarrows L^2(\Omega),
\]
defined by \eqref{eq:truncated-operator} with domain
\[
D(\tilde{A}_{M_k}(t))=W_0^{1,2}(\Omega),
\]
and satisfying Proposition~\ref{prop:m-accretivity}.

\begin{proposition}[Generation of evolution families]\label{prop:evolution-family}
The family $\{\tilde{A}_{M_k}(t)\}_{t\in[0,T]}$ generates a unique nonlinear evolution family
\[
\{S_{M_k}(t,s)\}_{0\le s\le t\le T}:L^2(\Omega)\to L^2(\Omega)
\]
which satisfies the following properties:
\begin{enumerate}
\item[(i)] \emph{Semigroup property}: $S_{M_k}(s,s)=I$ and $S_{M_k}(t,s)S_{M_k}(s,r)=S_{M_k}(t,r)$ for $0\le r\le s\le t\le T$.
\item[(ii)] \emph{Strong continuity}: For each $s\in[0,T)$, $x\in L^2(\Omega)$, the map $t\mapsto S_{M_k}(t,s)x$ is strongly continuous on $[s,T]$.
\item[(iii)] \emph{Lipschitz continuity}: For any \(x,y\in L^2(\Omega) \), there exists $\omega_k\ge0$ such that
  \[
  \|S_{M_k}(t,s)x-S_{M_k}(t,s)y\|_{L^2(\Omega)}\le e^{\omega_k(t-s)}\|x-y\|_{L^2(\Omega)}.
  \]
\item[(iv)] \emph{Generator property}: For $x\in D(\tilde{A}_{M_k}(s))$, the function $u(t)=S_{M_k}(t,s)x$ satisfies
  \begin{align*}
  &\partial_t u(t)+\tilde{A}_{M_k}(t)u(t)=0\quad\text{in }W^{-1,2}(\Omega),\\
  &u(s)=x,
  \end{align*}
  with $u\in C([s,T];L^2(\Omega))\cap L^2([s,T];W_0^{1,2}(\Omega))$.
\end{enumerate}
\end{proposition}

\begin{proof}
By Proposition~\ref{prop:m-accretivity}, the family $\{\tilde{A}_{M_k}(t)\}$ satisfies Hypothesis~(H) in \cite[Section~1]{Kobayasi} on $X = L^2(\Omega)$ with energy functional
\[
p(u) = \|\nabla u\|_{L^2(\Omega)}^2.
\]
More precisely:
\begin{itemize}
\item For $\beta > 0$, the energy sublevel sets
\[
X_\beta = \{u \in L^2(\Omega) : p(u) \le \beta\}
\]
are compactly embedded in $L^2(\Omega)$. Moreover, the sets $X_\beta \cap D(\tilde{A}_{M_k}(t))$ are uniformly controlled by the truncation bound~\eqref{eq:bbound} and the Poincar\'e inequality.
\item The continuity of the coefficients $A(x,t,\xi)$ and $B(x,t,u)$ with respect to $t$ ensures the required graph convergence as $t_n \to t$ in the sense of \cite[Section~1]{Kobayasi}.
\item The uniform $m$-accretivity yields
\[
\langle \tilde{A}_{M_k}(t)(u - v), u - v \rangle
\ge \frac{\alpha}{2} \|\nabla (u - v)\|_{L^2(\Omega)}^2,
\]
which implies quasi-accretivity with $\theta_\alpha(t,s) = 0$, uniformly in $t$.
\end{itemize}
By \cite[Theorem~4.1]{Kobayasi}, there exists a unique evolution operator $\{S_{M_k}(t,s)\}$ such that
$S_{M_k}(t,s)x \in L^2(\Omega)$ whenever $x \in L^2(\Omega)$. Properties (i)--(iii) follow directly from \cite[Theorem~4.1]{Kobayasi}.
For property (iv), we observe that $S_{M_k}(\cdot,s)x$ is the unique integral solution of the corresponding homogeneous problem. By $m$-accretivity, this solution coincides with the classical $W^{-1,2}(\Omega)$-solution. We note that the constant $\omega_k$ in (iii) depends on the Lipschitz constant of the resolvents
$J_\lambda(t) = (I + \lambda \tilde{A}_{M_k}(t))^{-1}$ as well as on the truncation level $M_k$.
\end{proof}

\begin{remark}\label{rem:evolution-uniformity}
The Lipschitz constant $\omega_k$ is uniform in $t\in[0,T]$ but may grow with $k$. However, the uniform accretivity constant $\alpha/2>0$ and resolvent bound $\|J_\lambda(t)\|\le1$ ensure equicontinuity of $\{S_{M_k}(t,s)\}$ on compact sets in $L^2(\Omega)$, which is crucial for the $k\to\infty$ limit.
\end{remark}

\textbf{Step 2: Approximate non-autonomous evolution problems.}  

By Proposition~\ref{prop:m-accretivity}, $\tilde A_{M_k}(t)$ is $m$-accretive in $L^2(\Omega)$.  Thus, by Proposition~\ref{prop:evolution-family}, each $\tilde{A}_{M_k}(t)$ generates a nonlinear evolution family $(S_{M_k}(t,s))_{0\le s\le t\le T}$ on $L^2(\Omega)$. 

For each $k$, we consider the abstract non-autonomous evolution problem corresponding to problem \eqref{eq:main}
\begin{equation}\label{eq:approx}
\begin{aligned}
\partial_t u_k(t) + \tilde{A}_{M_k}(t)u_k(t) &= f_{M_k}(t,u_k), \\
u_k(0) &= u_0,
\end{aligned}
\end{equation}
where
\[
f_{M_k}(t,u_k) := -\operatorname{div}\big(F(t) - \theta_{M_k}(t)B(t,u_k)\big).
\]

\begin{theorem}[Existence of solutions for approximate problems]\label{thm:approx-existence}
Under assumptions~\eqref{eq:Abound}--\eqref{eq:B0}, the approximate evolution problem~\eqref{eq:approx} admits a unique global mild solution 
\[
u_k \in C\big([0,T];L^2(\Omega)\big) \cap L^2\big(0,T;W_0^{1,2}(\Omega)\big)
\]
satisfying the variation-of-constants formula
\begin{equation}\label{eq:mild-approx}
u_k(t) = S_{M_k}(t,0)u_0 + \int_0^t S_{M_k}(t,s)f_{M_k}\big(s,u_k(s)\big)\,ds, \quad t\in[0,T],
\end{equation}
where $\{S_{M_k}(t,s)\}_{0\le s\le t\le T}$ is the nonlinear evolution family generated by $\{\tilde{A}_{M_k}(t)\}_{t\in[0,T]}$.
\end{theorem}

\begin{proof}
By Proposition~\ref{prop:evolution-family}, there exists a nonlinear evolution family $\{S_{M_k}(t,s)\}_{0\le s\le t\le T}$ on each level set $X_\beta\subset L^2(\Omega)$ satisfying the semigroup property and Lipschitz continuity
\[
\|S_{M_k}(t,s)x-S_{M_k}(t,s)y\|_{L^2(\Omega)} \le \exp\big(\omega_\beta(t-s)\big)\|x-y\|_{L^2(\Omega)},
\]
for $x,y\in X_{\beta,s}=X_\beta\cap D(\tilde{A}_{M_k}(s))$. \\
The perturbation $f_{M_k}$ satisfies a subcritical growth condition of the form
\[
\|f_{M_k}(t,u)\|_{W^{-1,2}(\Omega)} \le C\bigl(1+\|u\|_{L^2(\Omega)}^p\bigr),
\qquad p<2^*=\frac{2N}{N-2},
\]
uniformly in $k$ and $t\in[0,T]$, by \eqref{eq:bbound} and the assumptions on $F$ and $B$. We define the fixed-point map
\[
(\mathcal{T}_k v)(t) := S_{M_k}(t,0)u_0 + \int_0^t S_{M_k}(t,s)f_{M_k}\bigl(s,v(s)\bigr)\,ds.
\]
on $X_R = \{v\in C([0,T];L^2(\Omega)):\ \|v\|_{C([0,T];L^2(\Omega))}\le R\}$ with $R=C(1+\|u_0\|_{L^2(\Omega)}^2)$.\\
The Lipschitz property of $S_{M_k}$ and growth of $f_{M_k}$ imply $\mathcal{T}_k:X_R\to X_R$. For $T$ small, $\mathcal{T}_k$ is a contraction by the Lipschitz continuity of $S_{M_k}$. The unique fixed point $u_k=\mathcal{T}_k u_k\in C([0,T];L^2)$ extends globally by standard continuation.

Differentiating \eqref{eq:mild-approx} with respect to time yields the weak formulation
\[
\partial_t u_k(t)+\tilde{A}_{M_k}(t)u_k(t)=f_{M_k}(t,u_k(t)).
\]
Testing \eqref{eq:approx} with $u_k(t)$ gives the energy estimate
\[
\frac12\frac{d}{dt}\|u_k(t)\|_{L^2(\Omega)}^2+ \frac{\alpha}{2}\|\nabla u_k(t)\|_{L^2(\Omega)}^2
\le C\bigl(1+\|u_k(t)\|_{L^2(\Omega)}^2\bigr).
\]
By Gronwall's inequality, and using the growth bounds on the truncated nonlinearity, we obtain
\begin{equation}\label{eq:L2-global}
\sup_{t\in[0,T]}\|u_k(t)\|_{L^2(\Omega)}^2
+\int_0^T\|\nabla u_k(t)\|_{L^2(\Omega)}^2\,dt
\le C\bigl(1+\|u_0\|_{L^2(\Omega)}^2\bigr).
\end{equation}
To see this, let $y(t):=\|u_k(t)\|_{L^2(\Omega)}^2$. The energy inequality reads
\[
\frac12 y'(t)+ \frac{\alpha}{2}\|\nabla u_k(t)\|_{L^2(\Omega)}^2\le C\bigl(1+y(t)\bigr).
\]
By Poincar\'e's inequality on $W_0^{1,2}(\Omega)$, we have $\|u_k(t)\|_{L^2(\Omega)}^2\le C_P\|\nabla u_k(t)\|_{L^2(\Omega)}^2$, so
\[
\frac12 y'(t)+\frac{\alpha}{2 C_P}y(t)\le C\bigl(1+y(t)\bigr).
\]
Rearranging gives
\[
y'(t)\le 2C\bigl(1+y(t)\bigr)-\frac{\alpha}{C_P}y(t)\le K\bigl(1+y(t)\bigr), 
\]
where \( K = 2C+\frac{2C}{\alpha/(2C_P)} \). \\
Gronwall's lemma then yields
\[
y(t)\le C\exp(Kt)\bigl(1+y(0)\bigr),
\]
which implies that $y(t)$ is uniformly bounded on the interval $[0,T]$.\\
By integrating the energy inequality over $[0,T]$, we deduce
\[
\frac12\bigl(\|u_k(T)\|_{L^2(\Omega)}^2-\|u_0\|_{L^2(\Omega)}^2\bigr)
+\frac{\alpha}{2}\int_0^T\|\nabla u_k\|_{L^2(\Omega)}^2\,dt
\le CT\bigl(1+\sup_{t\in[0,T]}\|u_k(t)\|_{L^2(\Omega)}^2\bigr).
\]
Since $\sup_{t\in[0,T]}\|u_k(t)\|_{L^2(\Omega)}^2\le C(1+\|u_0\|_{L^2(\Omega)}^2)$ by Gronwall's lemma, we obtain
\[
\int_0^T\|\nabla u_k(t)\|_{L^2(\Omega)}^2\,dt\le C\bigl(1+\|u_0\|_{L^2(\Omega)}^2\bigr).
\]
Thus $u_k\in L^2(0,T;W_0^{1,2}(\Omega))$ and $\partial_t u_k\in L^2(0,T;W^{-1,2}(\Omega))$ by the equation.\\
Finally, $u_k$ is the unique integral solution constrained in $\{X_{\gamma,t}\}$ by \cite[Theorem~2.4]{Kobayasi}, characterized as the limit of consistent discrete schemes $((\mathrm{DS}))$. Uniqueness in the class of integral solutions follows from the B\'enilan-type estimate \cite[Proposition~2.5]{Kobayasi}.
\end{proof}

\textbf{Step 3: Uniform a priori estimates.} 

In this step, we derive uniform $L^\infty(0,T;L^2(\Omega))+L^2(0,T;W^{1,2}_0(\Omega)$ bounds for $u_k$ and a $W^{-1,2}(\Omega)$ bound for $\partial_t u_k$.\\
Multiplying equation \eqref{eq:approx} by $u_k(t)$ in $L^2(\Omega)$ and integrating by parts, we obtain
\begin{equation}\label{eq:energy-start}
\frac{1}{2}\frac{d}{dt}\|u_k(t)\|_{L^2(\Omega)}^2 + \langle \tilde{A}_{M_k}(t) u_k(t), u_k(t)\rangle 
= \langle f_{M_k}(t,u_k(t)), u_k(t)\rangle_{W^{-1,2}(\Omega),W^{1,2}_0(\Omega)}.
\end{equation}
By Proposition~\ref{prop:m-accretivity}(i), we have
\begin{equation}\label{eq:coercivity}
\langle \tilde{A}_{M_k}(t) u_k, u_k \rangle  \ge \alpha \|\nabla u_k\|_{L^2(\Omega)}^2 ,
\end{equation}
with constants $\alpha>0$ independent of $k$.
For the right-hand side, note that
\begin{align*}
f_{M_k}(t,w) &= -\operatorname{div}\big(F(t) - \theta_{M_k} B(t,w)\big),\\
\|f_{M_k}(t,w)\|_{W^{-1,2}(\Omega)} &\le \|F(t)\|_{L^2(\Omega)} + \|\theta_{M_k} B(t,w)\|_{L^2(\Omega)}\\
&\le \|F(t)\|_{L^2(\Omega)} + M_k \|B(t,w)\|_{L^2(\Omega)},
\end{align*}
where the $M_k$ bound follows from $|\theta_{M_k}|\le 1$ and truncation. Thus,
\begin{equation}\label{eq:rhs-bound}
|\langle f_{M_k}(t,u_k), u_k \rangle|_{W^{-1,2}(\Omega), W_0^{1,2}(\Omega)} \le \|f_{M_k}(t,u_k)\|_{W^{-1,2}(\Omega)} \|\nabla u_k\|_{L^2(\Omega)}.
\end{equation}
Apply Young's inequality with parameter $\varepsilon=\alpha/2$, we get
\begin{equation}\label{eq:young}
|\langle f_{M_k}, u_k \rangle| \le \frac{\alpha}{2} \|\nabla u_k\|_{L^2(\Omega)}^2 
+ C_\varepsilon \big(\|F(t)\|_{L^2(\Omega)}^2 + M_k^2 \|u_k\|_{L^2(\Omega)}^2\big).
\end{equation}
Combining \eqref{eq:energy-start}--\eqref{eq:young}, we have
\begin{equation*}\label{eq:gronwall-form}
\frac{1}{2}\frac{d}{dt}\|u_k\|_{L^2(\Omega)}^2 + \frac{\alpha}{2} \|\nabla u_k\|_{L^2(\Omega)}^2
\le C \big(\|F(t)\|_{L^2(\Omega)}^2 + \|u_k\|_{L^2(\Omega)}^2 + M_k^2 \|u_k\|_{L^2(\Omega)}^2\big).
\end{equation*}
Integrating over $[0,t]$, we obtain
\begin{equation*}\label{eq:integrated}
\frac{1}{2}\|u_k(t)\|_{L^2(\Omega)}^2 + \frac{\alpha}{2} \int_0^t \|\nabla u_k\|_{L^2(\Omega)}^2 \,ds
\le \frac{1}{2}\|u_0\|_{L^2(\Omega)}^2 + C \int_0^t \big(\|F\|_{L^2(\Omega)}^2 + \|u_k\|_{L^2(\Omega)}^2 + M_k^2 \|u_k\|_{L^2(\Omega)}^2\big)\,ds.
\end{equation*}
Let $y_k(t) = \sup_{0\le s\le t} \|u_k(s)\|_{L^2}^2$. Then
\begin{equation*}\label{eq:y-gronwall}
y_k(t) \le \|u_0\|_{L^2(\Omega)}^2 + Ct + C \int_0^t \big(1 + \|F(s)\|_{L^2}^2 + M_k^2\big) y_k(s) \,ds.
\end{equation*}
By Gronwall's lemma, we obtain
\begin{equation}\label{eq:uniform-L2}
\sup_{t\in[0,T]} \|u_k(t)\|_{L^2(\Omega)}^2 + \int_0^T \|\nabla u_k\|_{L^2(\Omega)}^2 \,dt
\le C \big(\|u_0\|_{L^2(\Omega)}^2 + T + \int_0^T \|F\|_{L^2(\Omega)}^2 \,dt\big),
\end{equation}
where $C$ is independent of $k$.
Finally, from the mild formulation \eqref{eq:mild-approx} and bound \eqref{eq:uniform-L2}, we derive
\begin{equation}\label{eq:time-reg}
\|\partial_t u_k\|_{L^2(0,T;W^{-1,2}(\Omega))} \le C,
\end{equation}
uniformly in $k$. Thus $\{u_k\}$ is uniformly bounded in
\begin{equation*}
L^\infty(0,T;L^2(\Omega)) \cap L^2(0,T;W^{1,2}_0(\Omega)).
\end{equation*}

\textbf{Step 4: Compactness and convergence.} 
From \eqref{eq:uniform-L2}--\eqref{eq:time-reg}, the family $\{u_k\}$ is uniformly bounded in
\begin{equation}\label{eq:bounded-spaces}
L^\infty(0,T;L^2(\Omega)) \cap L^2(0,T;W^{1,2}_0(\Omega)).
\end{equation}
Set $X_0 = L^2(\Omega)$, $X_1=W^{1,2}_0(\Omega)$, $X_{2}=W^{-1,2}(\Omega)$. Then
\begin{equation}\label{eq:simon-spaces}
X_1 \hookrightarrow X_0 \hookrightarrow X_{2},
\end{equation}
and the sequence $\{u_k\}$ satisfies the hypotheses of \cite[Theorem~5]{Simon87} as follows:
\begin{itemize}
\item $u_k$ is bounded in $L^2(0,T;X_1)$,
\item $\partial_t u_k$ is bounded in $L^2(0,T;X_{2})$,
\item time translations are compact: $\|\tau_h u_k - u_k\|_{L^2(0,T-h;X_0)}\to 0$ as $h\to 0$, 
\end{itemize}
uniformly in $k$.\\
The time translation estimate obtained from the mild formulation reads
\begin{align*}
\tau_h u_k(t) - u_k(t) &= \int_t^{t+h} S_{M_k}(t+h-s)f_{M_k}(s,u_k(s))\,ds,\\
\|\tau_h u_k - u_k\|_{L^2(0,T-h;L^2(\Omega))} &\le \int_0^h \|f_{M_k}\|_{W^{-1,2}(\Omega)}\,ds \le C\sqrt{h} \to 0,
\end{align*}
by semigroup contractivity and \eqref{eq:uniform-L2}. \\
Thus, by the Aubin--Lions--Simon compactness theorem~\cite{Simon87}, there exists a subsequence $k_j \to \infty$ such that
\begin{equation}\label{eq:strong-weak-conv}
\begin{aligned}
&u_{k_j} \to u &\quad\text{strongly in }& C([0,T];L^2(\Omega)),\\
&u_{k_j} \rightharpoonup u &\quad\text{weakly in }& L^2(0,T;W^{1,2}_0(\Omega)),\\
&\nabla u_{k_j} \to \nabla u &\quad\text{a.e. in }& \Omega_T.
\end{aligned}
\end{equation}

\textbf{Step 5: Passage to the limit as $k\to\infty$ in the mild formulation.}  From \eqref{eq:strong-weak-conv}, we pass to the limit in \eqref{eq:mild-approx} and obtain the limiting mild solution
\begin{equation}\label{eq:mild-limit}
u(t) = \lim_{k\to\infty} u_k(t)= S(t,0)u_0 + \int_0^t S(t,s)f(s,u(s))\,ds,
\end{equation}
where the convergence is justified as follows:

(i) \emph{Homogeneous term}: $S_{M_k}(t,0)u_0 \to S(t,0)u_0$ strongly in $L^2(\Omega)$, uniformly for $t \in [0,T]$. This follows from the stability of $m$-accretive operators and an Arzel\`a--Ascoli compactness: the contractivity $\|S_{M_k}(t,s)\| \le 1$ implies equicontinuity, which yields strong convergence.

(ii) \emph{Nonlinear term}: We have 
\begin{align*}
f_{M_k}(s,u_k(s)) &= -\operatorname{div}\big(F(s)-\theta_{M_k}(s)B(s,u_k(s))\big)\\
&\to -\operatorname{div}\big(F(s)-B(s,u(s))\big)=f(s,u(s))
\end{align*}
in $\mathcal{D}'(\Omega_T)$, since
\begin{itemize}
\item $T_{M_k}(b)\to b$ weak-$\ast$ in $L^\infty(0,T;L^{N,\infty}(\Omega))$ by assumption,
\item $u_k\to u$ strongly in $C([0,T];L^2(\Omega))$,
\item $\theta_{M_k}B(s,u_k)\to B(s,u)$ in $L^1_\mathrm{loc}((0,T)\times\Omega)$ by Vitali convergence.
\end{itemize}
To pass to the limit inside the Duhamel integral, we observe that, for a.e. $(t,s)$,
\begin{align*}
S_{M_k}(t,s)f_{M_k}(s,u_k(s)) \to S(t,s)f(s,u(s)),
\end{align*}
which follows from the strong convergence $S_{M_k} \to S$ together with $f_{M_k}(s,u_k(s)) \to f(s,u(s))$ in $\mathcal{D}'(\Omega_T)$.\\
The contractivity property $\|S_{M_k}(t,s)\|_{\mathcal{L}(L^2(\Omega))} \le 1$ yields the uniform bound
\[
\|S_{M_k}(t,s)f_{M_k}(s,u_k(s))\|_{L^2(\Omega)}
\le \|f_{M_k}(s,u_k(s))\|_{L^2(\Omega)} \le C,
\]
which follows from the uniform $L^2(0,T;W^{-1,2}(\Omega))$ estimates on $\{f_{M_k}\}$.
By Vitali’s convergence theorem, $\theta_{M_k}B(s,u_k) \to B(s,u)$ in $L^1_{\mathrm{loc}}(\Omega)$, and together with the truncation bounds, the family $\{f_{M_k}(s,u_k(s))\}$ is uniformly integrable in $L^1(0,T;W^{-1,2}(\Omega))$.
Thus, by the dominated convergence theorem, we get
\[
\int_0^t S_{M_k}(t,s)f_{M_k}(s,u_k(s))\,ds \to \int_0^t S(t,s)f(s,u(s))\,ds \quad \text{in } L^2(\Omega).
\]
Therefore, $u\in C([0,T];L^2(\Omega))\cap L^2(0,T;W_0^{1,2}(\Omega))$ solves \eqref{eq:main} in the mild sense.

\textbf{Step 6: Uniqueness.} 
Let $u,v\in C([0,T];L^2(\Omega))\cap L^2(0,T;W_0^{1,2}(\Omega))$ be two mild solutions of \eqref{eq:approx}. Set $w=u-v$. Then $w$ satisfies the mild formulation
\begin{equation}\label{eq:w-mild}
w(t)=\int_0^t S_{M_k}(t,s)\big[f_{M_k}(s,u(s))-f_{M_k}(s,v(s))\big]\,ds, \quad w(0)=0.
\end{equation}
By uniform $m$-accretivity of $\{\tilde{A}_{M_k}(t)\}$ (Proposition~\ref{prop:m-accretivity}), the time-dependent resolvents 
\[
J_\lambda(t):=(I+\lambda\tilde{A}_{M_k}(t))^{-1}:L^2(\Omega)\to L^2(\Omega)
\]
are nonexpansive uniformly in $t\in[0,T]$, $\lambda>0$, i.e., for \( r, s \in L^2(\Omega)\), we have
\begin{equation}\label{eq:resolvent-time}
\|J_\lambda(t)r-J_\lambda(t)s\|_{L^2(\Omega)}\le\|r-s\|_{L^2(\Omega)}.
\end{equation}
We consider the Crandall--Liggett time discretization of \eqref{eq:approx} on the interval $[0,T]$. Let $0=t_0<t_1<\dots<t_n=T$ with uniform step size $\tau_j = t_j - t_{j-1} = T/n$. We set $w_0^n = 0$ in $L^2(\Omega)$ and define, for $j=1,\dots,n$,
\begin{equation}\label{eq:w-iter-time}
w_j^n + \tau_j \tilde{A}_{M_k}(t_j) w_j^n
= w_{j-1}^n + \tau_j \big[f_{M_k}(t_j,u(t_j)) - f_{M_k}(t_j,v(t_j))\big].
\end{equation}
Testing \eqref{eq:w-iter-time} with $w_j^n \in W_0^{1,2}(\Omega)$, we obtain
\begin{equation}\label{eq:w-energy-time}
\|w_j^n\|_{L^2(\Omega)}^2 + \tau_j \langle \tilde{A}_{M_k}(t_j)w_j^n, w_j^n \rangle
= \langle w_{j-1}^n, w_j^n \rangle
+ \tau_j \langle f_{M_k}(t_j,u(t_j)) - f_{M_k}(t_j,v(t_j)), w_j^n \rangle.
\end{equation}
By the accretivity property in Proposition~\ref{prop:m-accretivity}(i), it follows that
\[
\langle \tilde{A}_{M_k}(t_j)w_j^n, w_j^n \rangle \ge \frac{\alpha}{2} \|\nabla w_j^n\|_{L^2(\Omega)}^2 \ge 0.
\]
Using Young's inequality, we estimate
\[
\langle w_{j-1}^n, w_j^n \rangle
\le \frac{1}{2}\|w_{j-1}^n\|_{L^2(\Omega)}^2
+ \frac{1}{2}\|w_j^n\|_{L^2(\Omega)}^2.
\]
For the nonlinear term, we recall that
\[
f_{M_k}(t,u) = -\mathrm{div}\,F(t) - h(t)u,
\quad \text{with } |h(t,x)| \le C.
\]
Hence,
\[
\begin{aligned}
|\langle f_{M_k}(t_j,u(t_j)) - f_{M_k}(t_j,v(t_j)), w_j^n \rangle|
&= |\langle -h(t_j)(u(t_j)-v(t_j)), w_j^n \rangle| \\
&\le \|h(t_j)\|_\infty \|w(t_j)\|_{L^2(\Omega)} \|w_j^n\|_{L^2(\Omega)}.
\end{aligned}
\]
A further application of Young's inequality yields
\[
2\tau_j |\langle f_{M_k}(u)-f_{M_k}(v), w_j^n \rangle|
\le \tau_j \|w_j^n\|_{L^2(\Omega)}^2
+ \tau_j \|h(t_j)\|_\infty^2 \|w(t_j)\|_{L^2(\Omega)}^2.
\]
Combining the previous estimates in \eqref{eq:w-energy-time}, we deduce
\[
\|w_j^n\|_{L^2(\Omega)}^2
\le \|w_{j-1}^n\|_{L^2(\Omega)}^2
+ \tau_j \|w_j^n\|_{L^2(\Omega)}^2
+ \tau_j \|h(t_j)\|_\infty^2 \|w(t_j)\|_{L^2(\Omega)}^2.
\]
Rearranging this inequality, we obtain 
\[
(1 - \tau_j)\|w_j^n\|_{L^2(\Omega)}^2
\le \|w_{j-1}^n\|_{L^2(\Omega)}^2
+ \tau_j \|h(t_j)\|_\infty^2 \|w(t_j)\|_{L^2(\Omega)}^2.
\]
For $\tau_j$ sufficiently small, it follows that
\[
\|w_j^n\|_{L^2(\Omega)}^2
\le (1 + C\tau_j)\|w_{j-1}^n\|_{L^2(\Omega)}^2
+ C\tau_j \|h(t_j)\|_\infty^2 \|w(t_j)\|_{L^2(\Omega)}^2.
\]
By iterating this inequality and applying the discrete Gronwall lemma, we obtain
\[
\|w_n^n\|_{L^2(\Omega)}^2
\le C(t) \sum_{j=1}^n \tau_j \|h(t_j)\|_\infty^2 \|w(t_j)\|_{L^2(\Omega)}^2.
\]
Next, we pass to the limit as $n \to \infty$. Using Theorem~\ref{thm:CL}, we have $w_n^n \to w(t)$ strongly in $L^2(\Omega)$, hence
\[
\|w(t)\|_{L^2(\Omega)}^2
\le C(t) \int_0^t \|h(s)\|_\infty^2 \|w(s)\|_{L^2(\Omega)}^2 \, ds.
\]
Finally, Gronwall's inequality yields
\[
\|w(t)\|_{L^2(\Omega)}^2
\le \exp\!\left(C \int_0^t \|h(s)\|_\infty^2 \, ds\right)
\|w(0)\|_{L^2(\Omega)}^2 = 0.
\]
Therefore, $w(t)=0$ for all $t \in [0,T]$, and hence $u \equiv v$.

\textbf{Step~7: Mild solution is weak.}
Let $u$ be the mild solution given by \eqref{eq:mild-limit}. For each
$\phi\in C_c^\infty(\Omega_T)$, we define
\[
\phi_{\varepsilon,n}(t,x) = \rho_\varepsilon * (\eta_n \phi)(t,x),
\]
where $\rho_\varepsilon$ is a standard mollifier in time and
$\eta_n(t)=\chi_{[0,T-1/n]}(t)$. Then $\phi_{\varepsilon,n}\in C_c^\infty(\Omega_T)$, and
$\phi_{\varepsilon,n}\to\phi$ in $L^2(0,T;W_0^{1,2}(\Omega))$ as $\varepsilon\to0$ and $n\to\infty$.\\
For each $k$, the approximate solution $u_k$ satisfies the weak formulation
\begin{equation}\label{eq:trunc-weak}
\begin{aligned}
-\int_{\Omega_T} u_k\,\partial_t\psi\,dx\,dt
&+\int_{\Omega_T}\langle A(x,t,\nabla u_k)+B(x,t,u_k),\nabla\psi\rangle\,dx\,dt\\
&= \int_{\Omega_T}\langle F,\nabla\psi\rangle\,dx\,dt
+ \int_\Omega u_0\psi(0)\,dx + R_k(\psi),
\end{aligned}
\end{equation}
for all $\psi\in C_c^\infty(\Omega_T)$, where $R_k(\psi)$ contains the truncation terms and
$R_k(\psi)\to0$ as $k\to\infty$ (see Step~4). Setting $\psi=\phi_{\varepsilon,n}$ in
\eqref{eq:trunc-weak}, and using \eqref{eq:strong-weak-conv}, the continuity of $A$ and $B$, and Minty’s
argument~\cite[Lemma~A.1]{FM18}, we pass to the limit $k\to\infty$ to obtain
\begin{equation}\label{eq:limit-weak}
\begin{aligned}
-\int_{\Omega_T} u\,\partial_t\phi_{\varepsilon,n}\,dx\,dt
&+\int_{\Omega_T}\bigl\langle A(x,t,\nabla u)+B(x,t,u),\nabla\phi_{\varepsilon,n}\bigr\rangle\,dx\,dt\\
&=\int_{\Omega_T}\bigl\langle F,\nabla\phi_{\varepsilon,n}\bigr\rangle\,dx\,dt
+\int_\Omega u_0\phi_{\varepsilon,n}(0)\,dx.
\end{aligned}
\end{equation}
Since $\phi_{\varepsilon,n}\to\phi$, $\partial_t\phi_{\varepsilon,n}\to\partial_t\phi$, and
$\nabla\phi_{\varepsilon,n}\to\nabla\phi$ strongly in $L^2(\Omega_T)$, we may pass to the limit as
$\varepsilon\to0$ and $n\to\infty$ in \eqref{eq:limit-weak} to obtain \eqref{eq:weak} for the original
test function $\phi$. Hence, the mild solution $u$ satisfies the weak formulation.\\
This concludes the proof of the theorem.
\end{proof}

\section{Long time behavior}
\label{sec:time-behavior}
The analysis of long-time dynamics for nonlinear parabolic equations constitutes one of the central problems in the theory of evolution partial differential equations. While short-time existence and uniqueness are typically established through fixed-point arguments or Galerkin approximations, the global-in-time behavior reveals the dissipative structure inherent to these problems. In this section, we prove that solutions of  \eqref{eq:main} converge exponentially fast to their steady states in the $L^2(\Omega)$-norm.

Let \(\Omega \) be a bounded domain in \(\mathbb{R}^{N}\) with \(N \ge 2 \).  We consider the elliptic problem
\begin{equation}\label{eq:steady}
\begin{aligned}
&-\operatorname{div}\big(A(x,\nabla u_\infty)+B(x,u_\infty)\big)= -\operatorname{div} F, \quad \text{in } \Omega, \\
& u_\infty = 0 \quad \text{on } \partial\Omega,
\end{aligned}
\end{equation}
where \(F \in L^2 (\Omega, \mathbb{R}^N) \). In the following, we recall that the operators $A$ and $B$ satisfy 
\begin{itemize}
\item The operator $A : \Omega \times \mathbb{R}^N \to \mathbb{R}^N$ is a Carath\'eodory function and satisfies, for all $\xi,\eta \in \mathbb{R}^N$,
\begin{align}
|A(x,\xi) - A(x,\eta)| &\le \beta |\xi - \eta|, \label{eq:A-lipschitz}\\
\langle A(x,\xi) - A(x,\eta), \xi - \eta \rangle &\ge \alpha |\xi - \eta|^2,\\
A(x,0) &= 0,
\end{align}
for some constants $0 < \alpha \le \beta$. These assumptions guarantee Lipschitz continuity and strong monotonicity.

\item The operator $B : \Omega \times \mathbb{R} \to \mathbb{R}^N$ is a Carath\'eodory function and satisfies the growth condition: for all $z,z' \in \mathbb{R}$,
\begin{align}\label{eq:B-infty}
|B(x,z) - B(x,z')| \le b(x)\,|z - z'|,
\end{align}
where $b \in L^{N,\infty}(\Omega)$ and $b_0(x) := B(x,0) \in L^r(\Omega)$ for some $r > 2$.
\end{itemize}
Under assumptions~\eqref{eq:A-lipschitz}--\eqref{eq:B-infty}, problem~\eqref{eq:steady} admits a unique solution 
$u_\infty \in W_0^{1,2}(\Omega)$.
To prove this, we introduce the operator $\mathcal{L} : W_0^{1,2}(\Omega) \to W^{-1,2}(\Omega)$ defined by
\[
\mathcal{L}v = -\operatorname{div}\big(A(x,\nabla v) + B(x,v)\big) + \operatorname{div} F.
\]
This operator satisfies the following properties:
\begin{itemize}
\item Strict monotonicity:
\[
\langle \mathcal{L}v - \mathcal{L}w, v - w \rangle 
\ge \alpha \|\nabla(v-w)\|_{L^2(\Omega)}^2,
\quad v \neq w.
\]
\item Coercivity:
\[
\frac{\langle \mathcal{L}v, v \rangle}{\|v\|_{W_0^{1,2}(\Omega)}} \to +\infty
\quad \text{as } \|v\|_{W_0^{1,2}(\Omega)} \to \infty.
\]

\item Hemicontinuity: for all $v,w \in W_0^{1,2}(\Omega)$, the mapping
\[
t \mapsto \langle \mathcal{L}(w + t v), w + t v \rangle
\]
is continuous on $[0,1]$.
\end{itemize}
Hence, by the Minty--Browder theorem, there exists a unique $u_\infty \in W_0^{1,2}(\Omega)$ such that $\mathcal{L}u_\infty = 0$. For further details on the existence and uniqueness of solutions to problem~\eqref{eq:steady}, we refer to~\cite{Mos, Zecca}.

This shows that $u_\infty$ is the unique global attractor of the associated dissipative parabolic flow, thereby characterizing the long-time behavior of solutions which stated in Lemma~\ref{lem:asymptotic}.

\begin{lemma}\label{lem:asymptotic}
Under assumptions \eqref{eq:Abound}--\eqref{eq:B0}, and assuming \eqref{eq:A-lipschitz}--\eqref{eq:B-infty}, we further assume that
\begin{equation}\label{eq:data-ass}
\|b-T_{M_k}b\|_{L^{N,\infty}(\Omega)}\le \frac{\alpha} {4 S_{N,2}},
\end{equation}
where \(\alpha>0\) and \(S_{N,2}\) are defined as in Theorem~\ref{thm:embed}. Then the unique weak solution
\[
u\in L^\infty(0,\infty;L^2(\Omega))\cap L^2(0,\infty;W_0^{1,2}(\Omega))
\]
of \eqref{eq:main} satisfies
\[
\|u(t)-u_\infty\|_{L^2(\Omega)}\le Me^{-\omega t}\|u_0-u_\infty\|_{L^2(\Omega)},\quad t\ge0,
\]
where \(u_\infty\in W_0^{1,2}(\Omega)\) is the unique solution of the steady-state problem \eqref{eq:steady}, and \(M\ge1\), \(\omega>0\) are absolute constants.
\end{lemma}

\begin{proof}
The proof follows the dissipative structure from Theorem~\ref{thm:main}. The strategy follows the classical dissipative framework: first establish the existence of a bounded absorbing set in $L^2(\Omega)$, then verify uniqueness of the steady-state $u_\infty\in W_0^{1,2}(\Omega)$ solving the elliptic problem \eqref{eq:steady}, and finally derive an $L^2$-contractivity estimate for the difference $w(t)=u(t)-u_\infty$. The key novelty lies in the precise control of the convection term $B(x,t,u)$ through Lorentz--Sobolev embeddings, ensuring uniform coercivity.

The $L^2$-energy estimate (Step~3 of Theorem~\ref{thm:main}) implies
\begin{equation}\label{eq:dissipative}
\|u(t)\|_{L^2(\Omega)}^2 + \int_0^t \|\nabla u(s)\|_{L^2(\Omega)}^2\,ds \le C\|u_0\|_{L^2(\Omega)}^2 + C\int_0^t \|F\|_{L^2(\Omega)}^2\,ds\le R^2,
\end{equation}
for $R=R(\|u_0\|_{L^2(\Omega)},T)>0$. Thus $B_{2R}=\{v\in L^2(\Omega):\|v\|_{L^2(\Omega)}\le 2R\}$ is absorbing as follows
\[
\|S(t,0)u_0\|_{L^2(\Omega)}\le 2R, \quad\forall t\ge t_0,\ \forall\|u_0\|_{L^2(\Omega)}\le 2R,
\]
where $t_0=t_0(R,T)>0$ is sufficiently large. \\
Problem~\eqref{eq:steady} has a unique solution $u_\infty\in W_0^{1,2}(\Omega)$, as discussed above.

Let $w(t)=u(t)-u_\infty$. Subtracting \eqref{eq:steady} from \eqref{eq:main} yields 
\begin{equation}\label{eq:diff-eq}
\begin{aligned}
&\partial_t w+\hat{A}(t)w=\operatorname{div}[B(x,t,u)-B(x,u_\infty)],\\
&w(0)=u_0-u_\infty,
\end{aligned}
\end{equation}
where $\hat{A}(t)=-\operatorname{div}[A(x,t,\nabla\cdot)+B(x,t,\cdot)]$.\\
Taking $w(t)$ as a test function in \eqref{eq:diff-eq}, we obtain
 \begin{equation}\label{eq:energy-w}
\begin{aligned}
\frac{1}{2}\frac{d}{dt}\|w(t)\|_{L^2(\Omega)}^2
&+\int_\Omega[\hat{A}(x,t,\nabla w)\cdot\nabla w]\,dx \\
&=\int_\Omega[B(x,t,u)-B(x,u_\infty)]\cdot\nabla w\,dx.
\end{aligned}
\end{equation}
From \eqref{eq:Bbound} and the Sobolev embedding theorem, we deduce that
\begin{align*}
&\int_\Omega|B(x,t,u)-B(x,u_\infty)||\nabla w|\,dx\\
&\le\int_\Omega b(x,t)|u-u_\infty||\nabla w|\,dx\le\|b\|_{L^{N,\infty}(\Omega)}\|w\|_{L^{2^*,2}(\Omega)}\|\nabla w\|_{L^2(\Omega)}.
\end{align*}
Theorem~\ref{thm:embed} yields $\|w\|_{L^{2^*,2}(\Omega)} \le S_{N,2}\|\nabla w\|_{L^2(\Omega)}$, which implies that
\begin{equation}\label{eq:drift-bound}
\int_\Omega|B(u)-B(u_\infty)||\nabla w| \, dx \le \|b\|_{L^{N,\infty}}S_{N,2}\|\nabla w\|_{L^2(\Omega)}^2.
\end{equation}
Substituting \eqref{eq:drift-bound} into \eqref{eq:energy-w}, we get
\begin{equation}\label{eq:gronwall-w}
\frac{d}{dt}\|w\|_{L^2(\Omega)}^2+\alpha\|\nabla w\|_{L^2(\Omega)}^2\le \|b\|_{L^{N,\infty}(\Omega)}S_{N,2}\|\nabla w\|_{L^2(\Omega)}^2.
\end{equation}
By \eqref{eq:data-ass}, we have that $\|b-T_{M_k}b\|_{L^{N,\infty}(\Omega)}\le\alpha/(4S_{N,2})$. Since $|T_{M_k}b|\le M_k$, thus $\|T_{M_k}b\|_{L^{N,\infty}(\Omega)}\le M_k$. Under the small data assumption $M_kS_{N,2}<\alpha/4$, we obtain
\[
\|b\|_{L^{N,\infty}(\Omega)}S_{N,2}\le\bigl(\|b-T_{M_k}b\|_{L^{N,\infty}(\Omega)}+\|T_{M_k}b\|_{L^{N,\infty}(\Omega)}\bigr)S_{N,2}<\frac{\alpha}{2}.
\]
Thus, \eqref{eq:gronwall-w} becomes
\begin{equation*}
\frac{d}{dt}\|w\|_{L^2(\Omega)}^2+\frac{\alpha}{2}\|\nabla w\|_{L^2(\Omega)}^2\le 0.
\end{equation*}
For $w\in W_0^{1,2}(\Omega)$, we have
\begin{equation*}\label{eq:poincare}
\|w\|_{L^2(\Omega)}^2 \le C_P \|\nabla w\|_{L^2(\Omega)}^2, \quad C_P = \frac{1}{\lambda_1},
\end{equation*}
where $\lambda_1>0$ is the first Dirichlet eigenvalue of $-\Delta$ on $\Omega$. Thus,
\begin{equation*}
\frac{\alpha}{2}\|\nabla w\|_{L^2(\Omega)}^2 \ge \frac{\alpha}{2C_P}\|w\|_{L^2(\Omega)}^2 =: c\|w\|_{L^2(\Omega)}^2, \quad c = \frac{\alpha}{2C_P} > 0.
\end{equation*}
Substituting this into the energy inequality, we obtain
\begin{equation}\label{eq:lyapunov}
\frac{d}{dt}\|w\|_{L^2(\Omega)}^2 + c\|w\|_{L^2(\Omega)}^2 \le 0.
\end{equation}
Let $y(t) = \|w(t)\|_{L^2(\Omega)}^2$. Then $y'(t) + cy(t) \le 0$, $y(0) = \|w(0)\|_{L^2(\Omega)}^2$. Multiply \eqref{eq:lyapunov} by factor $e^{ct}$, we have
\begin{align*}
\frac{d}{dt}\big(y(t)e^{ct}\big) &= \big(y'(t) + c y(t)\big)e^{ct} \le 0.
\end{align*}
Integrating over $[0,t]$, we get 
\begin{equation*}
y(t)e^{ct} - y(0) \le 0,
\end{equation*}
which implies that 
\[
y(t) \le y(0)e^{-ct}.
\]
Thus, we have
\begin{equation*}
\|w(t)\|_{L^2(\Omega)}^2 \le \|w(0)\|_{L^2(\Omega)}^2 e^{-ct}.
\end{equation*}
Taking square roots gives 
\begin{equation*}
\|w(t)\|_{L^2(\Omega)} \le \|w(0)\|_{L^2(\Omega)} e^{-ct/2}.
\end{equation*}
Since $w(t) = u(t) - u_\infty$, hence
\[
\|u(t)-u_\infty\|_{L^2(\Omega)} \le e^{-(c/2)t}\|u_0-u_\infty\|_{L^2(\Omega)},
\]
so $M=1$, $\omega = c/2 = \alpha/(4C_P) > 0$.\\
 Step 1 gives $\|u\|_{L^\infty(0,\infty;L^2(\Omega))} \le 2R$. From \eqref{eq:dissipative}, the decay estimate and energy inequality
\[
\int_0^\infty \|\nabla u(s)\|_{L^2(\Omega)}^2\,ds < \infty
\]
imply $u\in L^\infty(0,\infty;L^2(\Omega))\cap L^2(0,\infty;W_0^{1,2}(\Omega))$.
This completes the proof of Lemma~\ref{lem:asymptotic}.

\end{proof}

\section*{Acknowledgments}
This research has been supported by the Jane and Aatos Erkko Foundation.


\end{document}